\documentclass[reqno,12pt]{amsart}
\usepackage{mathrsfs}
\usepackage{amscd}
\usepackage{amsmath}
\usepackage{latexsym}
\usepackage{amsfonts}
\usepackage{amssymb}
\usepackage{amsthm}
\usepackage{graphicx}

\def\R{\mathbb{R}}

\numberwithin{equation}{section}

\newtheorem{thm}{Theorem}[section]
\newtheorem{lem}{Lemma}[section]

\newtheorem{prop}{Proposition}[section]

\newtheorem{remark}{Remark}[section]

\makeatletter
\newcommand{\Extend}[5]{\ext@arrow0099{\arrowfill@#1#2#3}{#4}{#5}}
\makeatother

\begin{document}

\setcounter{page}{1}

\title[Global existences and asymptotic behavior for semilinear heat equation]{Global existence and asymptotic behavior for $L^2$-critical semilinear heat equation in negative Sobolev spaces}

\author{Avy Soffer}
\address{Rutgers University\\
Department of Mathematics\\
110 Frelinghuysen Rd.\\
Piscataway, NJ, 08854, USA\\}
\address{Department of Mathematics\\
Hubei Key Laboratory of Mathematical Science\\
Central China Normal University\\
Wuhan 430079, China.\\}
\email{soffer@math.rutgers.edu}
\thanks{}

\author{Yifei Wu}
\address{Center for Applied Mathematics\\
Tianjin University\\
Tianjin 300072, China}
\email{yerfmath@gmail.com}
\thanks{}

\author{Xiaohua Yao}
\address{Department of Mathematics\\
Hubei Key Laboratory of Mathematical Science\\
Central China Normal University\\
Wuhan 430079, China.\\}
\email{yaoxiaohua@mail.ccnu.edu.cn}


\date{}

\maketitle

\begin{abstract}\noindent
In this paper, we consider the global Cauchy  problem for the $L^2$-critical semilinear heat equations
$
\partial_t h=\Delta h\pm |h|^{\frac4d}h,
$
with $h(0,x)=h_0$, where  $h$ is an unknown real function defined on $ \R^+\times\R^d$.  In most of the studies on this
subject, the initial data $h_0$ belongs to Lebesgue spaces $L^p(\R^d)$ for some $p\ge 2$ or
to subcritical Sobolev space $H^{s}(\R^d)$ with $s>0$. {\it First,} we  prove that  there exists some positive constant $\gamma_0$ depending on $d$, such that the Cauchy problem is locally and globally well-posed for any initial data $h_0$ which is radial, supported away from the origin
and in the negative Sobolev space $\dot H^{-\gamma_0}(\R^d)$. In particular, it leads to local and global existences of the solutions to Cauchy problem considered above for the initial data in a proper subspace of $L^p(\R^d)$ with some $p<2$. {\it Secondly,} the sharp asymptotic behavior of the solutions  ( i.e. $L^2$-decay estimates ) as $t\to +\infty$ are obtained  with arbitrary large initial data $h_0\in \dot H^{-\gamma_0}(\R^d)$ in the defocusing case  and  in the focusing case with suitably small initial data $h_0$.

\end{abstract}


 \baselineskip=20pt

\section{Introduction}
\subsection{Main results}
Consider the initial value problem for a semilinear heat equation:
\begin{equation}\label{heat'}
   \left\{ \aligned
    &\partial_t h=\Delta h\pm |h|^{\gamma-1}h, \\
    &h(0,x)=h_0(x), \quad x\in \R^d,
   \endaligned
  \right.
 \end{equation}
where  $h(t,x)$ is an unknown real function defined on $ \R^+\times\R^d$, $d\ge 2$, $\gamma>1$. The positive sign ``+" in nonlinear term of \eqref{heat'} denotes focusing source,   and the negative sign ``-" denotes the defocusing one. The Cauchy problem \eqref{heat'} has been extensively studied in Lebesgue space $L^p(\R^d)$ by many people, see e.g. \cite{Bre-Caz, Bre-Fre, Bre-Pel, Ga-Va, Giga, Ma-Mer, Mer, Miao-Zhang, Monn, Ni-Sacks, Ribaud, Ruf-Terrano, Tan,Weiss0, Weiss1, Weiss2}. The equation
 enjoys an interesting property of scaling invariance
\begin{equation}\label{scaling}
h_\lambda(t,x):=\lambda^{2/(\gamma-1)}h(\lambda^2t, \lambda x), \ h_\lambda(0,x):= \lambda^{2/(\gamma-1)}h_0(\lambda x),  \ \lambda>0;
\end{equation}
that is, if $h(t,x)$ is the solution of heat equation \eqref{heat'}, then $h_\lambda(t,x)$ also solves the equationwith the scaling data $\lambda^{2/\gamma}h_0(\lambda x)$.
An important fact is that Lebesgue space $L^{p_c}(\R^d)$ with $p_c=\frac{d(\gamma-1)}{2}$ is the only one invariant under the same scaling transform: \begin{equation}\label{scaling2}h_0(x)\mapsto\lambda^{2/(\gamma-1)}h_0(\lambda x).\end{equation}
If we consider the initial data $h_0\in L^p(\R^d)$, then the scaling index
$p_c=\frac{d(\gamma-1)}{2}$
plays a critical role in the local/global well-posedness of \eqref{heat'}. Roughly speaking, one can divide the dynamics of \eqref{heat'} into the following three different regimes:  (A) {\it the subcritical case $p>p_c $},  (B) {\it the critical case $p=p_c$}, (C) {\it the supercritical case $p<p_c$}.  More specifically,

In the subcritical case (A) ( i.e. $p> p_c$ ), Weissler \cite{Weiss0} first proved the local existence and unconditional uniqueness of solution for initial values in $L^p$ with $p>p_c$ and $p\ge\gamma$. In \cite{Weiss1}, the local existence and conditional uniqueness were further considered in the solution space $C([0, T); L^p(\R^d))\cap L^\infty_{loc}((0,T]; L^\infty(\R^d))$ for initial data in $L^p$ when $p>p_c$ and $p\ge\gamma$. Later, Brezis and Cazenave \cite{Bre-Caz} proved the unconditional uniqueness of Weissler's solution mentioned above.

In the critical case (B) ( i.e. $p=p_c$ ), the local existence and conditional uniqueness were obtained  by  Giga \cite{Giga} and Weissler \cite{Weiss1} for initial data in $L^p$ with $p>1$ and $p=p_c$. The unconditional uniqueness was proved in the cases $p=p_c>\gamma$ by Brezis and Cazenave \cite{Bre-Caz}.  In {\it double critical case} $p=p_c=\gamma$ ( i.e. $p=\gamma=\frac{d}{d-2}$ ), the local conditional wellposedness of the problem \eqref{heat'} was due to Weissler in \cite{Weiss1}, but the unconditional uniqueness fails, see Ni-Sacks \cite{Ni-Sacks}, Terraneo \cite{Terran}.

In the supercritical case $(C)$  ( i.e. $p<p_c$ ), it seems that  there exists no local solution in any reasonable sense for some initial data $h_0\in L^p(\R^d)$ with $p<p_c$. In particular, in the focusing case, there exists a nonnegative function $h_0\in L^p(\R^d)$ such that  \eqref{heat'} does not admit any nonnegative classical $L^p$-solution in $[0,T)$ for any $T>0$, see e.g. Brezis and Cabr\'e \cite{Bre-Cab}, Brezis and Cazenave \cite{Bre-Caz}, Haraux-Weissler \cite{Har-Weissler} and  Weissler \cite {Weiss1, Weiss2}. Also, we refer the readers to the book Quitnner-Souplet \cite{Qui-Soup} for many related topics and references.

For singular initial values, like $|x|^{-\gamma}$ or $\delta(x)$, the behavior of the solutions have also  been extensively studied, see \cite{Cazenave-Dickstein-Naumkin-Weissler, Hattab-Slim-Weissler, Tayachi-Weissler} and the references therein.

In this paper, we  are mainly concerned with global existences and asymptotic behavior of solutions  for {\it a large class of ~``supercritical" initial data $h_0\in L^p(\R^d)$ with some $p<p_c,$ } and more generally, a class of initial data in $ \dot H^{-\gamma}$ for some  $\gamma>0$ depending on $d$. For simplicity, we only consider  the Cauchy
problem for the $L^2$-critical semilinear heat equations,
\begin{equation}\label{heat}
   \left\{ \aligned
    &\partial_t h=\Delta h+\mu|h|^{\frac4d}h,\ \ \mu=\pm 1, \\
    &h(0,x)=h_0(x), \quad x\in \R^d.
   \endaligned
  \right.
 \end{equation}
That is, $p_c=2$\ ( i.e $\gamma=1+\frac4d $ ), we will first prove that  there exists some positive constant $\gamma_0$ depending on $d$, such that the Cauchy problem is locally and globally well-posed for any initial data $h_0$ is radial, supported away from the origin
and in the negative Sobolev space $\dot H^{-\gamma_0}(\R^d)$. It  includes certain $L^p$-spaces with $p<2 \ ( p_c=2 )$ as initial data subspace (see Remark \ref{rem:01} below). Moreover,  we also establish the sharp  point-wise decay estimates for the solution $h(t,x)$ in $L^2(\R^d)$ as  $ t \to +\infty$ for the initial datum $h_0 \in \dot H^{-\gamma_0}(\R^d)$ which is large in the defocusing case and  suitably small  in the focusing case data $h_0$.

The first results are given as follows:
\begin{thm} \label{thm:main1}
Let $ \mu=\pm1$, $d\ge 2$ and
 \begin{equation}\label{index}
\gamma_0\in\aligned
    &\Big[0, \frac{d-1}{d+2}\Big).\\
   \endaligned
 \end{equation}
Suppose that $h_0\in \dot H^{-\gamma_0}(\R^d)$ is a radial initial data satisfying
$
\mbox{supp } h_0\subset \{x:|x|\ge1\}.
$
Then there exists a time $\delta>0$ depending on $\|h_0\|_{H^{-\gamma_0}}$ and $d$ such that  a unique strong solution
 \begin{equation}\label{local solution}
 h\in C([0,\delta); L^2(\R^d)+\dot H^{-\gamma_0}(\R^d))\cap L^\frac{2(2+d)}{d}_{tx}([0,\delta]\times\R^d)
 \end{equation}
 to the equation \eqref{heat} exists with the initial data $h_0$.

Furthermore, if
$d>4$, then the following canonical decomposition holds in the sense that there exists  an unconditionally unique function $w$ in $C([0,\delta], L^2(\R^d))$ such that the solution $h$ of \eqref{local solution} can be stated as
 \begin{equation}\label{struct}
h=e^{t\Delta}h_0+w.
 \end{equation}

\end{thm}

Several further remarks on Theorem \ref{thm:main1} above are given as follows:
\begin{remark}\label{rem:01}
 In Theorem \ref{thm:main1}, if the support condition on $h_0$ is replaced by the $ \mbox{supp } h_0\subset\{x:|x|\ge \eta\}$ for any fixed $\eta>0$, then our results have no essential changes besides the lifespan constant $\delta$ surly depending on $\eta$.
 Moreover, note that if $h_0\in L^p$ for some $p\in (1,2)$, then there exists some $\epsilon>0$ such that $h_0\in\dot H^{-\varepsilon}(\R^d) $ and
$$\|h_0\|_{\dot H^{-\varepsilon}(\R^d)}\lesssim \|h_0\|_{L^p(\R^d)}$$
by the Sobolev embedding estimate. Thus,  Theorem \ref{thm:main1}  shows that the solution $h$ of the equation \eqref{heat}
always exists locally, and the solution $h$ of the defocusing equation \eqref{heat}
exists globally for any radial and supported away from zero initial datum $h_0$  in $L^p(\R^d),$ with certain supercritical index $p\in \Big(\frac{2(d+2)}{d+3},2\Big)$ for $d\ge 2$.
\end{remark}

\begin{remark}\label{rem:04}
It seems that the restriction $d>4$ is necessary for unconditional uniqueness. In fact, when $d=4$, the uniqueness problem is related to the ``double critical'' case ( i.e. $p=p_c=\gamma=\frac{d}{d-2}=2$ ). It is well-known that the unconditional uniqueness fails by Ni-Sacks \cite{Ni-Sacks} and Brezis and Cazenave \cite{Bre-Caz}.
 \end{remark}

We emphasize that  the key part of the proof of Theorem \ref{thm:main1} is to prove the following time-space estimate
\begin{equation}\label{Strichartz type}
\big\| e^{t\Delta}h_0\big\|_{L^{2+(4/d)}(\R^+\times\R^d)}\lesssim \||\nabla|^{-\gamma_0}h_0\|_{L^2(\R^d)}=\|h_0\|_{\dot H^{-\gamma_0}(\R^d)},
\end{equation}
if the $h_0$ and $\gamma_0$ satisfy with the conditions of Theorem \ref{thm:main1}.  This is compared with the following standard estimate ( also see  \eqref{nabla-L2'} below ):
\begin{equation*}
\big\| e^{t\Delta}h_0\big\|_{L^{2+(4/d)}(\R^+\times\R^d)}\lesssim\|h_0\|_{L^2(\R^d)},
\end{equation*}
which plays a central role in the well-posedness of the solution to the Cauchy problem \eqref{heat} in the critical space $L^2(\R^d)$. The estimate \eqref{Strichartz type}  actually indicates  that it is possible to study $L^2$ critical-Cauchy problem \eqref{heat} using the following initial space:
$$Z=\Big\{h_0\in \mathcal{S}'(\R^d); \ \ (t,x)\mapsto e^{t\Delta}h_0(x)\in L^{2+(4/d)}(\R^+\times\R^d) \Big\}$$
equipped with the norm
$$
\|h_0\|_Z=\| e^{t\Delta}h_0\|_{L^{2+(4/d)}(\R^+\times\R^d)}.
$$
We refer to Giga \cite{Giga} for the proof framework.
Note that the data space $Z$ is  a critical space in the scaling sense \eqref{scaling2},  which especially includes the $L^2(\R^d)$ and these initial datum $h_0$ of Theorem \ref{thm:main1} by the estimate \eqref{Strichartz type}.
However, we remark that,  working on $Z$ directly  does not give more than the estimates expressed in Theorem \ref{thm:main1}, which will serve the subsequent conclusions below.  So instead,  we decompose the initial data $h_0$ into the sum of $v_0$ and $w_0$ such that $v_0\in \dot H^{-\gamma_0}(\R^d)$ and $w_0\in L^2(\R^d);$ and then evolve the linear heat flow $v_L(t)$ for $v_0$ in the space $ \dot H^{-\gamma_0}(\R^d),$ and work on the nonlinear heat flow $w(t)$ for $w_0$ in $L^2(\R^d)$, where the estimate \eqref{Strichartz type} is indispensable for the present arguments. From the local behavior of the two flows we finally conclude the desired solution $h(t)$.

Next, we describe further results on {\it the global existence } of the solutions for the Cauchy
problem \eqref{heat}, with any large initial data $h_0\in \dot H^{-\varepsilon}(\R^d)$ in the defocusing case,  and in the focusing case with suitably small initial data $h_0$. However, we  remark that  for the focusing critical solution with a large initial data, the long time behaviors may have many different scenarios, which display possibly the blow-up or steady states, see e.g., Giga and Kohn \cite{Gi-Ko}. In the sequel, we do not consider these interesting phenomenons in the focusing cases. In fact, one can refer to Kavian \cite{Kav},
Kawanago \cite{Kawa}, Quittner \cite{Quit} and the book Quitnner-Souplet \cite{Qui-Soup} for many related results and a large number of references therein. Specifically, our global result  is as follows:

  \begin{thm} \label{thm:main2}Let $h_0$, $d$ and $\gamma_0$ satisfy the conditions in Theorem \ref{thm:main1}.  In addition,  if $\mu=-1$ ( the defocusing case ),  or if $\mu=1$ ( the focusing case ) and  $\|h_0\|_{\dot H^{-\gamma_0}(\R^d)}$ is small enough, then the solution $h(t)$ in \eqref{local solution} is global in time.
\end{thm}

Lastly, we are concerned with the asymptotic behavior of the solution. Based on the following linear flow estimate:
$$
\big\|e^{t\Delta}h_0\big\|_{L^2(\R^d)}\lesssim t^{-\frac{\gamma_0}{2}}\|h_0\|_{\dot H^{-\gamma_0}},\ \ t>0,
$$
we expect that the nonlinear flow obeys the same decay estimate. There is a body of works on the decay estimates for the parabolic systems. For example, for the heat equations \eqref{heat'}, it has an optimal universal $L^\infty$-decay estimate with rate $t^{-\frac d4},$ when the time is $t$ away from zero, see for example \cite{Galaktionov-Va-book};  for the Navier-Stokes equations, it was proved by Schonbek \cite{Schonbek1, Schonbek2} that when $u_0\in L^1(\R^3)\cap  L^2(\R^3)$, there is $L^2$-decay with rate of $t^{-\frac34}$. As the third main result of our paper, we show that under the hypotheses:
$$
h_0\in \dot H^{-\gamma_0}, \quad \mbox{ and }
\quad
\big\|h(t)\big\|_{L^{2+(4/d)}(\R^+\times\R^d)}< +\infty,
$$
(which are the essential conditions we need in the following proof),
the nonlinear solution has the same $L^2$-decay estimate as the linear flow. Our long-time decay result  is

\begin{thm} \label{thm:main3} Let $h(t)$ be the global solution obtained in Theorem \ref{thm:main2}. Then  the following pointwise decay estimate holds:
 \begin{equation}\label{decay}
\|h(t)\|_{L^2}\lesssim C(\|h_0\|_{\dot H^{-\gamma_0}})\ t^{-\frac{\gamma_0}{2}},  \ \ t>0,
\end{equation}
where $C\big(\|h_0\|_{\dot H^{-\gamma_0}}\big)=O\big(\|h_0\|_{\dot H^{-\gamma_0}}^\kappa \big)$ for some $\kappa>0$.
\end{thm}

 It  is worth noting  that the decay estimate \eqref{decay} holds in the defocusing case with large initial data $h_0,$ and   the decay rate $\gamma_0$ is not necessarily  small as the condition presented in Theorem \ref{thm:main1}. Therefore, the standard bootstrap argument does not work directly in our case. Due  to the low regularity and without additional spatial decay of the initial data (like $u_0\in L^p, p<2$), the existing argument (e.g., \cite{Schonbek1, Schonbek2})does not seem to work in this situation. In this paper, some new idea is needed to overcome the crucial technical obstacles for large initial data in the defocusing case. In particular, a time-dependent frequency cut-off estimate of $P_{\le N(t)}h(t)$  is introduced to gain the smallness via two steps in different directions, where $P_{<N}$ is a Littewood-Paley projection of $L^{2}(\R^d)$. To clear our argument, we give the sketch of its proof.


\noindent {\bf Sketch of proof of Theorem \ref{thm:main3}: }

The proof is divided into the following three steps.

Step 1,  a time-dependent low frequency estimate. More precisely, we obtain that for any $N(t)>0$ and any $\alpha\in[0,\gamma_0)$,
\begin{align}\label{Low-Est}
\big\|P_{\le N(t)}h(t)\big\|_{L^2(\R^d)}
\le
t^{-\frac{\alpha}{2}}C\big(\|h_0\|_{\dot H^{-\gamma_0}}\big)\Big(1+N(t)t^{\frac12\alpha}+N(t)t^{\frac{1}2}\ \|h\|_{Y_\alpha(T)}\Big),
\end{align}
where
$\|h\|_{Y_\alpha(T)}=\sup\limits_{t\in[0,T]}\big(t^\frac{\alpha}{2}\|h(t)\|_{L^2(\R^d)}\big)$.
One may note that the estimates above obey the scaling invariance.

Step 2, a weak $L^2$-estimate. Applying the estimate \eqref{Low-Est}, we can prove that for some  $0<\alpha<<1$,
 \begin{equation}\label{decay-weak}
\|h(t)\|_{L^2}\lesssim C(\|h_0\|_{\dot H^{-\gamma_0}}) t^{-\frac{\alpha}{2}},  \ \ t>1.
\end{equation}
Indeed, we have the following modified energy estimate,
\begin{align*}
\frac{d}{dt}\|h(t)\|_{L^2}^2+N(t)^2\big\|h(t)\big\|_{L^2}^2\le N(t)^2\left\|P_{\le N(t)}h(t)\right\|_{L^2}^2.
\end{align*}
Then this combined with the estimate \eqref{Low-Est} and choosing the function $N(t)$ such that $N(t)^2=2\alpha t^{-1}$,  yields that
\begin{align*}
 \|h\|_{Y_\alpha(T)}^2
\le
&
C\big(\|h_0\|_{\dot H^{-\gamma_0}}\big)+2\alpha C\big(\|h_0\|_{\dot H^{-\gamma_0}}\big)\|h\|_{Y_\alpha(T)}^2.
\end{align*}
Thanks to the smallness from $\alpha$,  we obtain \eqref{decay-weak}.

Step 3, $L^2$-estimate with the decaying rate of $t^{-\frac{\gamma_0}{2}}$.   Based on the weak decay estimate above, we can improve  the estimate in Step 1 as
\begin{align}\label{Low-Est-improved}
\big\|P_{\le N(t)}h(t)\big\|_{L^2(\R^2)}
\le
t^{-\frac{\gamma_0}{2}}C\big(\|h_0\|_{\dot H^{-\gamma_0}}\big)\Big(1+N(t)t^{\frac12\gamma_0}+N(t)t^{\frac{1}2}\>t^{-\frac12\alpha} \|h\|_{Y_{\gamma_0}(T)}\Big).
\end{align}
It gains a  $t^{-\frac12\alpha}$ in  the front of the last term. In particular,  we can get a smallness from it when $t$ is large.  This  combined the mass estimate finally leads to the claimed optimal result by the bootstrap.

\subsection{The organization of the paper}Here we will describe the organization of  the paper and outline the key ideas of proofs.

In Section 2, we will list some preliminaries used in the paper,  including Littlewood-Paley multipliers and serval time-space estimates of linear flow $e^{t\Delta}$.

In Section 3, we are  devoted to giving the proof of Theorem \ref{thm:main1}.
Specifically, we decompose the initial data $h_0$ into the sum of $v_0$ and $w_0$ such that $v_0\in \dot H^{-\gamma_0}(\R^d)$ and $w_0\in L^2(\R^d)$, and then evolute the linear heat flow $v_L(t)$ for $v_0$ in the space $ \dot H^{-\gamma_0}(\R^d)$ and do the nonlinear heat flow $w(t)$ for $w_0$ in $L^2(\R^d)$, where the estimate \eqref{Strichartz type} is still indispensable in the present arguments. From the two flow we finally conclude the desired solution $h(t)$,

In Section 4, we will show Theorem \ref{thm:main2}. In focusing case ($\mu=1$), we  need the smallness assumption on  $\|h_0\|_{\dot H^{-\gamma_0}(\R^d)}$. But in the defocusing case ($\mu=-1$), the additional restriction the smallness can be removed.


\section{Preliminary}

\subsection{Littlewood-Paley multipliers and related inequalities}

Throughout this paper, we write $A\lesssim B$ to signify
that there exists a constant $c$ such that $A\leq cB$, while we denote $A\sim B$ when
$A\lesssim B\lesssim A$. We first define the {\it  Littlewood-Paley projection multiplier}. Let $\varphi(\xi)$ be a fixed real-valued radially symmetric bump function adapted to the ball $\{\xi\in \R^d: \ |\xi|\le 2\}$, which equals 1 on the ball $\{\xi\in \R^d: \ |\xi|\le 1\}$. Define a {\it dyadic number} to any number $N\in 2^{\mathbf{Z}}$ of the form $N=2^j$ where $j\in \mathbf{Z}$ ( the integer set). For each dyadic number $N$, we define the Fourier multipliers
$$\widehat{P_{\le N}f}(\xi):=\varphi(\xi/N)\hat{f}(\xi), \ \ \widehat{P_{N}f}(\xi):=\varphi(\xi/N)-\varphi(2\xi/N)\hat{f}(\xi), $$
where $\hat{f}$ denotes the Fourier transform of $f$.  Moreover, define $P_{>N}=I-P_{\le N}$ and  $P_{<N}=P_{\le N}-P_N$, etc. In particular, we have the telescoping expansion:
$$P_{\le N}=\sum_{M\le N} P_M f; \ \ P_{>N}=\sum_{M> N} P_M f$$
where $M$ ranges over dyadic numbers. It was well-known that the Littlewood-Paley operators satisfy the following useful {\it Bernstein inequalities} with $s>0$ and $1\le p\le q\le\infty$ ( see e.g. Tao \cite{tao} ):
$$\|P_{\ge N}f\|_{L^p_x(\R^d)}\lesssim N^{-s}\||\nabla|^s P_{\ge N}f\|_{L_x^p(\R^d)}, \ \||\nabla|^s P_{\le N}f\|_{L^p_x(\R^d)}\lesssim N^{s}\|| P_{\le N}f\|_{L_x^p(\R^d)};$$
$$ \||\nabla|^{\pm s}  P_{ N}f\|_{L^p_x(\R^d)}\sim N^{\pm s}\|| P_{ N}f\|_{L_x^p(\R^d)};$$
$$\|P_{N}f\|_{L^q_x(\R^d)}\lesssim N^{(\frac{d}{p}-\frac{d}{q})}\| f\|_{L_x^p(\R^d)}, \ \|P_{\le N}f\|_{L^q_x(\R^d)}\lesssim N^{(\frac{d}{p}-\frac{d}{q})}\|f\|_{L_x^p(\R^d)};$$

Moreover, we also have the following {\it mismatch estimate}, see e.g. \cite{LiZh-APDE}.
\begin{lem}[Mismatch estimates]\label{lem:mismatch}
Let $\phi_1$ and $\phi_2$ be smooth functions obeying
$$
|\phi_j| \leq 1 \quad \mbox{ and }\quad \mbox{dist}(\emph{supp}
\phi_1,\, \emph{supp} \phi_2 ) \geq A,
$$
for some large constant $A$.  Then for $m>0$, $N\ge 1$ and $1\leq p\leq
q\leq \infty$,
\begin{align*}
\bigl\| \phi_1  P_{\le  N} (\phi_2 f)
\bigr\|_{L^q_x(\R^d)}=\bigl\| \phi_1  P_{\ge N} (\phi_2 f)
\bigr\|_{L^q_x(\R^d)}
    \lesssim_m A^{-m+\frac dq-\frac dp}N^{-m} \|\phi_2 f\|_{L^p_x(\R^d)}.
\end{align*}
\end{lem}

\subsection{Space-time estimates of the linear heat equation}Let $e^{t\Delta}$ denote the heat semigroup on $\mathbb{R}^d$. Then for suitable function $f$,  $e^{t\Delta}f$ solves the linear heat equation
$$
\partial_t h=\Delta h, \ \ h(0,x)=f(x),\ t>0, \ x\in\R^d,
$$
and the solution satisfies the following fundamental space-time estimates:
\begin{lem}\label{lem:Stri} Let $f\in L^p(\R^d)$ for $1\le p\le \infty$, then
\begin{align}\label{lem:Lp-bound}
\big\|e^{t\Delta}f\big\|_{L^\infty_tL^p_x(\R^+\times\R^d)}\lesssim \|f\|_{L^p(\R^d)}.
\end{align}
Moreover, let $I\subset \R^+$, then for $f\in L^2(\R^d)$ and $F\in L^\frac{2(2+d)}{d+4}_{tx}(\R^+\times\R^d)$,
\begin{align}
\big\|\nabla e^{t\Delta}f\big\|_{L^2_{tx}(\R^+\times\R^d)}&\lesssim \|f\|_{L^2(\R^d)};\label{nabla-L2}
\end{align}
\begin{align}\label{nabla-L2'}
\big\| e^{t\Delta}f\big\|_{L^{\frac{2(2+d)}{d}}_{tx}(\R^+\times\R^d)}&\lesssim \|f\|_{L^2(\R^d)};
\end{align}
\begin{align}\label{Lpq}
\Big\|\int_0^t e^{(t-s)\Delta}F(s)ds\Big\|_{L^\infty_tL^2_x\ \cap \ L^\frac{2(2+d)}{d}_{tx}\cap\ L^2_t\dot H^1_x(I\times\R^d)}\lesssim \|F\|_{L^\frac{2(2+d)}{d+4}_{tx}(\R^+\times\R^d)}.
\end{align}
\end{lem}

We can give some remarks on the inequalities $\eqref{lem:Lp-bound}-\eqref{Lpq}$ above as follows:

 (i). The estimate \eqref{lem:Lp-bound} is classical and  immediately follows from the Young inequality by the following heat kernel integral:
$$(e^{t\Delta}f)(x)=( 4\pi t)^{-d/2}\int_{\R^d}e^{-|x-y|^2/4t}f(y)dy, \ t>0.$$
More generally, for all $1\le p\le q\le\infty$, the following (decay) estimate hold:
\begin{align}
\|e^{t\Delta}f\|_{L^q(\R^d)}\lesssim t^{\frac{d}{2}(\frac{1}{q}-\frac{1}{p})} \|f\|_{L^p(\R^d)}, \  \ t>0.\label{Linear-decay}
\end{align}

 (ii). The estimate \eqref{nabla-L2} is equivalent to  a kind of square-function inequality on $L^2(\R^d)$, which can be reformulated as
$$\Big\| \Big( \int_0^\infty |\sqrt{t}\nabla e^{t\Delta}f|^2\frac{dt}{t}\Big)^{\frac{1}{2}}\Big\|_{L^2(\R^d)}\lesssim \|f\|_{L^2(\R^d)},$$
which follows directly by the Plancherel theorem, and also holds in the $L^p(\R^d)$ for $1<p<\infty$ ( see e.g. Stein\cite[p. 27-46]{Stein} ).

(iii). The estimate \eqref{nabla-L2'} can be obtained by interpolation between the \eqref{lem:Lp-bound} and \eqref{nabla-L2}:
$$
\big\|e^{t\Delta}f\big\|_{L^\frac{2(2+d)}{d}_{tx}(\R^+\times\R^d)}\lesssim \big\|e^{t\Delta}f\big\|_{L^\infty_tL^2_x(\R^+\times\R^d)}^{\frac 2{d+2}}
\big\|\nabla e^{t\Delta}f\big\|_{L^2_{tx}(\R^+\times\R^d)}^{\frac d{d+2}}.
$$

(iv). The estimate \eqref{Lpq} consists of the three same type inequalities with the different norms $L^\infty_tL^2_x$, $ L^\frac{2(2+d)}{d}_{tx}$ and  $L^2_t\dot H^1_x$ on the left side. As shown in (iii) above, the second norm $ L^\frac{2(2+d)}{d}_{tx}$ can be controlled by interpolation between $L^\infty_tL^2_x$ and $L^2_t\dot H^1_x$. Due to the similarity of their  proofs, we  will give a  proof to the first one, which in turn is the special case of the following lemma.
 It is worth noting that the estimate is $L^2$-subcritical for $p<\infty$.
\begin{lem}\label{lem:Stri-1} Let $2\le p\le \infty$, and the pair $(p_1,r_1)$ satisfy
$$
\frac2{p_1}+\frac d{r_1}=\frac d2+2+\frac2p,\quad 1\le p_1\le 2,\quad 1<r_1\le2,
$$
then
$$
\Big\|\int_0^t e^{(t-s)\Delta}F(s)ds\Big\|_{L^p_tL^2_x(\R^+\times\R^d)}\lesssim \|F\|_{L_t^{p_1}L_x^{r_1}(\R^+\times\R^d)}.
$$
\end{lem}
\begin{proof}
By Plancherel's theorem, it is equivalent that
 \begin{align}\label{lpq'}\Big\|\int_0^t e^{-(t-s)|\xi|^2}\widehat{F}(\xi, s)ds \Big\|_{L^p_tL^2_\xi(\R^+\times\R^d)}\lesssim \|F\|_{L_t^{p_1}L_x^{r_1}(\R^+\times\R^d)}.
 \end{align}
Since  by the Young inequality of the convolution on $\R^+$, for any $1\le p_1\le p\le \infty$,
$$\Big\|\int_0^t e^{-(t-s)|\xi|^2}\widehat{F}(\xi, s)ds\Big\|_{L^p(\R^+)}\lesssim \Big\||\xi|^{-(\frac2p+\frac2{p_1'})}\widehat{F}(\xi,\cdot)\Big\|_{L_t^{p_1}(\R^+)}.$$
Note that $p_1\le 2\le p$, thus by Minkowski's inequality,  Plancherel's theorem, Sobolev's embedding, we obtain
 \begin{align*}
 \Big\|\int_0^t e^{-(t-s)|\xi|^2}\widehat{F}(\xi, s)ds \Big\|_{L^p_tL^2_\xi(\R^+\times\R^d)}&\lesssim \Big\||\xi|^{-(\frac2p+\frac2{p_1'})}\widehat{F}(\xi,\cdot)\Big\|_{L^2_\xi L_t^{p_1}(\R^+\times\R^d)}\\ &\lesssim \Big\||\nabla|^{-(\frac2p+\frac2{p_1'})}F\Big\|_{L_t^{p_1}L^2_x(\R^+\times\R^d)}\lesssim\|F\|_{L_t^{p_1}L_x^{r_1}(\R^+\times\R^d)}.
\end{align*}
which gives the desired estimate \eqref{lpq'}.
\end{proof}

Finally, we also need the following maximal $L^p$-regularity result for the heat flow. See \cite{LR-Book} for example.

\begin{lem}\label{lem:max-Lp} Let $p\in (1,\infty),q\in (1,\infty)$, and let $T\in (0,\infty]$, then the operator $A$ defined by $$f(t,x)\mapsto \int_0^t e^{(t-s)\Delta}\Delta f(s,\cdot)\,ds$$ is bounded from $L^p((0,T),L^q(\R^d))$ to $L^p((0,T),L^q(\R^d))$.
\end{lem}
\section{Proof of Theorem \ref{thm:main1}}

In this section, we will divide several subsections to finish the proof of  Theorem \ref{thm:main1}.
\subsection{A supercritical estimate of $e^{t\Delta}$} Let us begin with the following radial Sobolev embedding estimates, see e.g. \cite{TaViZh}.
\begin{lem}\label{lem:radial-Sob}
Let $\alpha,q,p,s$ be the parameters which satisfy
$$
\alpha>-\frac dq;\quad \frac1q\le \frac1p\le \frac1q+s;\quad 1\le p,q\le \infty; \quad 0<s<d
$$
with
$$
\alpha+s=d(\frac1p-\frac1q).
$$
Moreover, let at most one of the following equalities hold:
$$
p=1,\quad p=\infty,\quad q=1,\quad q=\infty,\quad \frac1p=\frac1q+s.
$$
Then the radial Sobolev embedding inequality holds:
\begin{align*}
\big\||x|^\alpha u\big\|_{L^q(\R^d)}\lesssim \big\||\nabla|^su\big\|_{L^p(\R^d)}.
\end{align*}

\end{lem}
\begin{lem}\label{prop:superest} For any $q>2$ and any $\delta\in \big(\frac12-\frac3q,1-\frac4q\big)$, suppose that the radial function $f\in H^\gamma(\R^d)$ satisfying
$$
\mbox{supp }f\subset \{x:|x|\ge1\},
$$
then
$$
\big\|e^{t\Delta}f\big\|_{L^q_{tx}(\R^+\times\R^d)}\lesssim \big\||\nabla|^\delta f\big\|_{L^2_{x}(\R^d)}.
$$
\end{lem}
\begin{proof} By Lemma \ref{lem:Stri}, we have
$$
\big\|e^{t\Delta}f\big\|_{L^\infty_{tx}(\R^+\times\R^d)}\lesssim \|f\|_{L^\infty(\R^d)}.
$$
Let $\alpha = \frac d2-s>0$ and $s\in (\frac12,1)$, then by Lemma \ref{lem:radial-Sob} we have
$$
\|f\|_{L^\infty(\R^d)}\lesssim\||x|^\alpha f\|_{L^\infty(\R^d)}\lesssim \big\||\nabla|^{s}f\big\|_{L^2(\R^d)},
$$
where the first inequality above has used the condition $\mbox{supp }f\subset \{x:|x|\ge1\}$.
Thus we get that
\begin{align}
\big\|e^{t\Delta}f\big\|_{L^\infty_{tx}(\R^+\times\R^d)}\lesssim \big\||\nabla|^{s}f\big\|_{L^2(\R^d)}.\label{12.44}
\end{align}
Interpolation between this last estimate and \eqref{nabla-L2}, gives our desired estimates.
\end{proof}

\subsection{A decomposition of the initial data $h_0$} We use $\chi_{\le a}$ for $a\in \R^+$ to denote the smooth function
\begin{align*}
\chi_{\le a}(x)=\left\{ \aligned
1, \ & |x|\le a,\\
0,    \ &|x|\ge \frac{11}{10} a,
\endaligned
  \right.
\end{align*}
and set $\chi_{\ge a}=1-\chi_{\le a}$.

Now write
\begin{align}\label{initial-decp}
h_0=v_0+w_0,
\end{align}
where
$$
v_0=\chi_{\ge \frac12}\big(P_{\ge N}h_0\big),\quad w_0=h_0-v_0.
$$

 \underline{First, we will claim that $w_0\in L^2(\R^d)$ and}
\begin{align}
\|w_0\|_{L^2(\R^d)}\lesssim N^{\gamma_0}\big\|h_0\big\|_{\dot H^{-\gamma_0}(\R^d)}.\label{eq:21.43}
\end{align}
Note that we can write that
$$w_0=\chi_{\le \frac12}\big(P_{\ge N}h_0\big)+P_{< N}h_0.$$
By the Bernstein estimate (see Section 2.1 ), we have that
$$
\big\|P_{<N}h_0\big\|_{L^2(\R^d)}\lesssim
N^{\gamma_0}\big\|h_0\big\|_{\dot  H^{-\gamma_0}(\R^d)}.
$$
This  combining with the following Lemma \ref{lem:h-mis},   gives the desired \eqref{eq:21.43}.
\begin{lem}\label{lem:h-mis}
Let $h_0$ be the function satisfying the hypothesis in Theorem \ref{thm:main1}, then
\begin{align}
\big\|\chi_{\le \frac12}\big(P_{\ge N}h_0\big)\big\|_{L^2(\R^d)}\lesssim
N^{-1}\big\|h_0\big\|_{\dot H^{-\gamma_0}(\R^d)}.
\end{align}
\end{lem}
\begin{proof}
By the support property of $h_0$, we may write
\begin{align}
\chi_{\le \frac12}&\big(P_{\ge N}h_0\big)
=\chi_{\le \frac12}\big(P_{\ge N}\chi_{\ge \frac{9}{10}}h_0\big)\notag\\
=&
\chi_{\le \frac12}\big(P_{\ge N}\chi_{\ge \frac{9}{10}}P_{\le 2N}h_0\big)
+\sum\limits_{M=4N}^\infty\chi_{\le \frac12}P_{\ge N}\big(\chi_{\ge \frac{9}{10}}P_Mh_0\big).\label{20.55}
\end{align}
By Lemma \ref{lem:mismatch} and Bernstein's inequality, we have
\begin{align}
\big\|\chi_{\le \frac12}\big(P_{\ge N}\chi_{\ge \frac{9}{10}}P_{\le 2N}h_0\big)
\big\|_{L^2(\R^d)}
&\lesssim
N^{-10}\big\|P_{\le 2N}h_0\big\|_{L^2(\R^d)}\notag\\
&\lesssim
N^{-1}\big\|h_0\big\|_{\dot  H^{-\gamma_0}(\R^d)}.
\label{20.55-I}
\end{align}
Moreover, since $P_{\ge N}=I-P_{< N}$ and $M>2N$, we obtain
\begin{align*}
\chi_{\le \frac12}&P_{\ge N}\big(\chi_{\ge \frac{9}{10}}P_Mh_0\big)
=-\chi_{\le \frac12}P_{< N}\big(\chi_{\ge \frac{9}{10}}P_Mh_0\big)\\
&=-\chi_{\le \frac12}P_{< N}\Big(P_{\ge\frac18M}\big(\chi_{\ge \frac{9}{10}}\big)P_Mh_0\Big),
\end{align*}
where $P_{\ge\frac18M}(\chi_{\ge \frac{9}{10}})$ denotes the high frequency truncation of the bump function $\chi_{\ge \frac{9}{10}}$.

Note that
\begin{align*}
\Big\|\chi_{\le \frac12}P_{<N}\Big(P_{\ge\frac18M}&\big(\chi_{\ge \frac{9}{10}}\big)P_Mh_0\Big)\Big\|_{L^2(\R^d)}
\lesssim \big\|P_{\ge\frac18M}\big(\chi_{\ge \frac{9}{10}}\big)\big\|_{L^\infty(\R^d)}\big\|P_Mh_0\big\|_{L^2(\R^d)}\\
\lesssim &M^{-2}\big\|\Delta P_{\ge\frac18M}\big(\chi_{\ge \frac{9}{10}}\big)\big\|_{L^\infty(\R^d)}\big\|P_Mh_0\big\|_{L^2(\R^d)}\\
\lesssim &M^{-1}\big\|h_0\big\|_{\dot H^{-\gamma_0}(\R^d)}.
\end{align*}
Hence, we have
\begin{align*}
\Big\|\chi_{\le \frac12}&P_{\ge N}\Big(\chi_{\ge \frac{9}{10}}P_Mh_0\Big)\Big\|_{L^2(\R^d)}
\lesssim M^{-1}\big\|h_0\big\|_{\dot H^{-\gamma_0}(\R^d)}.
\end{align*}
Therefore, taking summation, we obtain
\begin{align}
\sum\limits_{M=4N}^\infty\big\|\chi_{\le \frac12}P_{\ge N}\big(\chi_{\ge \frac{9}{10}}P_Mh_0\big)\big\|_{L^2(\R^d)}
\lesssim N^{-1}\big\|h_0\big\|_{\dot H^{-\gamma_0}(\R^d)}.
\label{20.55-II}
\end{align}
Inserting \eqref{20.55-I} and \eqref{20.55-II} into \eqref{20.55}, we prove the lemma.
\end{proof}

\underline{Second, we claim that $v_0\in H^{-\gamma_0}(\R^d)$ and}
\begin{align}
\big\|v_0\big\|_{\dot  H^{-\gamma_0}(\R^d)}\lesssim \big\|h_0\big\|_{\dot  H^{-\gamma_0}(\R^d)}.\label{11.56}
\end{align}
Indeed, since $v_0=h_0-\chi_{\le \frac12}\big(P_{\ge N}h_0\big)$, so it follows that
\begin{align*}
\big\|v_0\big\|_{\dot  H^{-\gamma_0}(\R^d)}
\lesssim &
\big\|h_0\big\|_{\dot  H^{-\gamma_0}(\R^d)}
+\big\|\chi_{\le \frac12}\big(P_{\ge N}h_0\big)\big\|_{\dot  H^{-\gamma_0}(\R^d)}.
\end{align*}
Hence, we only consider the latter term. By Sobolev's embedding and H\"{o}lder's inequality, we have
\begin{align*}
\big\|\chi_{\le \frac12}\big(P_{\ge N}h_0\big)\big\|_{\dot  H^{-\gamma_0}(\R^d)}\lesssim
\big\|\chi_{\le \frac12}\big(P_{\ge N}h_0\big)\big\|_{L^q(\R^d)}
\lesssim
\big\|\chi_{\le \frac12}\big(P_{\ge N}h_0\big)\big\|_{L^2(\R^d)}.
\end{align*}
where $1/2=1/q-\gamma_0/d $ with some $q<2$. Hence \eqref{11.56} follows from Lemma \ref{lem:h-mis}.

\subsection{The solutions associated with $v_0$ and $w_0$ }

We denote
$$
v_L(t)=e^{t\Delta}v_0,\ \ w(t)=h(t)-v_L(t).
$$

\underline{For the linear part $v_L(t)$},  it follows that by Plancherel's theorem and \eqref{11.56},
\begin{align}\label{linear part}
\big\|v_L(t)\big\|_{L^\infty_t \dot H^{-\gamma_0}_x(\R^+\times\R^d)}\lesssim \big\|v_0\big\|_{\dot  H^{-\gamma_0}(\R^d)}\lesssim \big\|h_0\big\|_{\dot  H^{-\gamma_0}(\R^d)}.
\end{align}
Hence $v_L(t)$ is globally existence in time and satisfies the following time-space estimate:

\begin{lem}\label{impro-est}
Let $\epsilon$ be a sufficiently small positive constant such that $0<\epsilon<\frac{d-1}{d+2}-\gamma_0$, then we have that
\begin{align}
\big\|v_L(t)\big\|_{L^\frac{2(2+d)}{d}_{tx}(\R^+\times\R^d)}\lesssim N^{-\frac{d-1}{d+2}+\gamma_0+\epsilon}\big\|h_0\big\|_{\dot  H^{-\gamma_0}(\R^d)}.\label{1.14}
\end{align}
\end{lem}

\begin{proof}Let $\gamma=-\frac{d-1}{d+2}+\epsilon$, then by Lemma \ref{prop:superest},
\begin{align*}
\big\|v_L(t)\big\|_{L^\frac{2(2+d)}{d}_{tx}(\R^+\times\R^d)}\lesssim \big\||\nabla|^{\gamma}\chi_{\ge \frac12}\big(P_{\ge N}h_0\big)\big\|_{L^2(\R^d)}.
\end{align*}
Note that
$$
\big\||\nabla|^{\gamma}\chi_{\ge \frac12}\big(P_{\ge N}h_0\big)\big\|_{L^2(\R^d)}
\le \big\||\nabla|^{\gamma}\big(P_{\ge N}h_0\big)\big\|_{L^2(\R^d)}+\big\||\nabla|^{\gamma}\chi_{\le \frac12}\big(P_{\ge N}h_0\big)\big\|_{L^2(\R^d)}.
$$
For the first term, since $\gamma<-\gamma_0$, by Bernstein's inequality,
$$
\big\||\nabla|^{\gamma}\big(P_{\ge N}h_0\big)\big\|_{L^2(\R^d)}
\lesssim
N^{\gamma+\gamma_0}\big\|h_0\big\|_{\dot H^{-\gamma_0}(\R^d)}.
$$
So we only need to estimate the second term. Let $q$ be the parameter satisfying
$$
\frac1q=\frac12-\frac\gamma d,
$$
then $q>1$.
Since $\gamma<0$, by Sobolev's and H\"older's inequalities,
$$
\big\||\nabla|^{\gamma}\chi_{\le \frac12}\big(P_{\ge N}h_0\big)\big\|_{L^2(\R^d)}
\lesssim \big\|\chi_{\le \frac12}\big(P_{\ge N}h_0\big)\big\|_{L^q(\R^d)}
\lesssim  \big\|\chi_{\le \frac12}\big(P_{\ge N}h_0\big)\big\|_{L^2(\R^d)}.
$$
Furthermore, by Lemma \ref{lem:h-mis},
$$
 \big\|\chi_{\le \frac12}\big(P_{\ge N}h_0\big)\big\|_{L^2(\R^d)}
 \lesssim N^{-1}\big\|h_0\big\|_{\dot H^{-\gamma_0}(\R^d)}.
$$
Combining the last two estimates above, we obtain
$$
\big\||\nabla|^{\gamma}\chi_{\le \frac12}\big(P_{\ge N}h_0\big)\big\|_{L^2(\R^d)}
\lesssim N^{-1}\big\|h_0\big\|_{\dot H^{-\gamma_0}(\R^d)}.
$$
This gives the desired \eqref{1.14}.
\end{proof}

\underline{For the nonlinear part $w(t)$},  it is clear that $w$ satisfies with the following Cauchy problem:
\begin{equation}\label{heat-w}
   \left\{ \aligned
    &\partial_t w=\Delta w\pm|w+v_L|^{\frac4d}(w+v_L), \\
    &w(0,x)=w_0(x)=h_0-v_0\in L^2(\R^d),
   \endaligned
  \right.
 \end{equation}
which is equivalent to the original Cauchy problem \eqref{heat}.  The following lemma is to establish the local well-posedness of the  \eqref{heat-w} and give global criterion of $w$.
\begin{lem}\label{lem:local} There exists $\delta>0$, such that for any $h_0$ satisfying the hypothesis in Theorem \ref{thm:main1} and $w_0=h_0-v_0$,  the Cauchy problem \eqref{heat-w} is well-posed on the time interval $[0,\delta]$, and
the solution $$w\in C_tL^2_x([0,\delta]\times\R^d)\cap L^\frac{2(2+d)}{d}_{tx}([0,\delta]\times\R^d)\cap L^2_t\dot H^1_x([0,\delta]\times\R^d).$$ Furthermore, let $T^*$ be the maximal lifespan, and  suppose that
$$w\in L^\frac{2(2+d)}{d}_{tx}([0,T^*)\times\R^d),$$
then $T^*=+\infty$. In particular, if $\|h_0\|_{\dot  H^{-\gamma_0}(\R^d)}\ll 1$, then $T^*=+\infty$ and
$$
\|w\|_{C_tL^2_x([0,+\infty]\times\R^d)}\lesssim\|h_0\|_{\dot  H^{-\gamma_0}(\R^d)} .
$$
\end{lem}
\begin{proof}
For local well-posedness, we only show that the solution $w\in L^\infty_tL^2_x([0,\delta]\times\R^d)\cap L^\frac{2(2+d)}{d}_{tx}([0,\delta]\times\R^d)\cap L^2_t\dot H^1_x([0,\delta]\times\R^d)$ for some $\delta>0$. Indeed, the local well-posedness with the lifespan $[0, \delta)$ is then followed by the standard fixed point argument.
By Duhamel's formula, we have
$$
w(t)=e^{t\Delta}w_0\pm\int_0^te^{(t-s)\Delta}|h(s)|^\frac4dh(s)\,ds.
$$
Then by Lemma \ref{lem:Stri}, for any $t_*\le \delta$,
\begin{align*}
\big\|w\big\|_{L^\frac{2(2+d)}{d}_{tx}([0,t_*]\times\R^d)}
\lesssim &\|e^{t\Delta}w_0\|_{L^\frac{2(2+d)}{d}_{tx}([0,t_*]\times\R^d)}+\big\||h|^\frac4dh\big\|_{L^\frac{2(2+d)}{d+4}_{tx}([0,t_*]\times\R^d)}\\
\lesssim &\|e^{t\Delta}w_0\|_{L^\frac{2(2+d)}{d}_{tx}([0,\delta]\times\R^d)}+\big\|h\big\|_{L^\frac{2(2+d)}{d}_{tx}([0,t_*]\times\R^d)}^{\frac4d+1}.
\end{align*}
Note that
$$
\big\|h\big\|_{L^\frac{2(2+d)}{d}_{tx}([0,t_*]\times\R^d)}\lesssim \big\|v_L\big\|_{L^\frac{2(2+d)}{d}_{tx}(\R^+\times\R^d)}
+\big\|w\big\|_{L^\frac{2(2+d)}{d}_{tx}([0,t_*]\times\R^d)},
$$
let $\eta_0=(\frac4d+1)\big(\frac{d-1}{d+2}-\gamma_0-\epsilon\big)>0$, then using \eqref{1.14}, we obtain
\begin{align*}
\big\|w\big\|_{L^\frac{2(2+d)}{d}_{tx}([0,t_*]\times\R^d)}
\lesssim &\|e^{t\Delta}w_0\|_{L^\frac{2(2+d)}{d}_{tx}([0,\delta]\times\R^d)}+N^{-\eta_0}\big\|h_0\big\|_{\dot  H^{-\gamma_0}(\R^d)}^{\frac4d+1}
+\big\|w\big\|_{L^\frac{2(2+d)}{d}_{tx}([0,t_*]\times\R^d)}^{\frac4d+1}.
\end{align*}
Noting that either $\|h_0\|_{\dot  H^{-\gamma_0}(\R^d)}\ll 1$, or choosing $\delta$ small enough and $N$ large enough, we have
$$
\|e^{t\Delta}w_0\|_{L^\frac{2(2+d)}{d}_{tx}([0,\delta]\times\R^d)}+N^{-\eta_0}\big\|h_0\big\|_{\dot  H^{-\gamma_0}(\R^d)}^{\frac4d+1}\ll1,
$$
then by the continuity argument, we get
\begin{align}
\big\|w\big\|_{L^\frac{2(2+d)}{d}_{tx}([0,\delta]\times\R^d)}\lesssim \|e^{t\Delta}w_0\|_{L^\frac{2(2+d)}{d}_{tx}([0,\delta]\times\R^d)}+N^{-\eta_0}\big\|h_0\big\|_{\dot  H^{-\gamma_0}(\R^d)}^{\frac4d+1}. \label{11.47}
\end{align}
Further, by Lemma \ref{lem:Stri} again,
\begin{align*}
\big\|w\big\|_{L^2_t\dot H^1_x([0,\delta]\times\R^d)}+\sup\limits_{t\in [0,\delta]}\big\|w\big\|_{L^2_x(\R^d)}
\lesssim &\|w_0\|_{L^2_x(\R^d)}+\big\||h|^\frac4dh\big\|_{L^\frac{2(2+d)}{d+4}_{tx}([0,\delta]\times\R^d)}\\
\lesssim &\|w_0\|_{L^2_x(\R^d)}+\big\|v_L\big\|_{L^\frac{2(2+d)}{d}_{tx}([0,\delta]\times\R^d)}^{\frac4d+1}
+\big\|w\big\|_{L^\frac{2(2+d)}{d}_{tx}([0,\delta]\times\R^d)}^{\frac4d+1}.
\end{align*}
Hence, using \eqref{1.14} and \eqref{11.47}, we obtain
\begin{align*}
\big\|w\big\|_{L^2_t\dot H^1_x([0,\delta]\times\R^d)}+\sup\limits_{t\in [0,\delta]}\big\|w\big\|_{L^2_x(\R^d)}
\le C,
\end{align*}
for some $C=C(N,\big\|h_0\big\|_{\dot  H^{-\gamma_0}(\R^d)})>0$.

Suppose that
$$w\in L^\frac{2(2+d)}{d}_{tx}([0,T^*)\times\R^d),$$
then if $T^*<+\infty$, we have
\begin{align*}
\big\|w(T^*)\big\|_{L^2_x(\R^d)}
\lesssim &\|e^{t\Delta}w_0\|_{L^\frac{2(2+d)}{d}_{tx}([0,T^*]\times\R^d)}+\big\||h|^\frac4dh\big\|_{L^\frac{2(2+d)}{d+4}_{tx}([0,T^*)\times\R^d)}\\
\lesssim &\|w_0\|_{L^2_{x}(\R^d)}+N^{-\eta_0}\big\|h_0\big\|_{\dot  H^{-\gamma_0}(\R^d)}^{\frac4d+1}
+\big\|w\big\|_{L^\frac{2(2+d)}{d}_{tx}([0,T^*)\times\R^d)}^{\frac4d+1}.
\end{align*}
Hence, $w$ exists on $[0,T^*]$, and $w(T^*)\in L^2(\R^d)$. Hence, using the local theory obtained before from time $T^*$, the lifespan can be extended to $T^*+\delta$, this is contradicted with the definition of the  maximal lifespan $T^*$. Hence, $T^*=+\infty$.
\end{proof}

\subsection{The proof of Theorem \ref{thm:main1}} We first prove the existences of $h(t)$ in Theorem \ref{thm:main1}. In fact,  let $h(t)=w(t)+v_L(t)$. Then from Lemma \ref{impro-est} and Lemma \ref{lem:local}, there exists a  $\delta>0$ depending on $\|h_0\|_{H^{-\gamma_0}}$ and $d$ such that
 \begin{equation}\label{local solution-II}
 h(t)\in C([0,\delta); L^2(\R^d)+\dot H^{-\gamma_0}(\R^d))\cap L^\frac{2(2+d)}{d}_{tx}([0,\delta]\times\R^d)
 \end{equation}
 and $h(t)$ satisfies the Cauchy equation \eqref{heat}  with the initial data $h_0$. Thus we have finished the local well-posedness of Cauchy equation \eqref{heat}.

In the sequel, we turn to show the unconditionally unique decomposition $\eqref{struct}$ of $h(t)$. Inspired by \cite{Monn}, we
apply the maximal $L^p$-regularity of the heat flow $e^{t\Delta}$.
Let $h_1,h_2$ be two any solutions of \eqref{heat} with the same initial data $h_0$, and satisfy
$$
h_1=e^{s\Delta}h_0+w_1;\quad h_2=e^{s\Delta}h_0+w_2,
$$
where $w_1,\ w_2\in C([0,\delta], L^2(\R^d))$. By the Duhamel formula, we have
\begin{align*}
w_1(t)=&\int_0^t e^{(t-s)\Delta}|e^{s\Delta}h_0+w_1|^{\frac4d}\big(e^{s\Delta}h_0+w_1\big)\,ds;\\
w_2(t)=&\int_0^t e^{(t-s)\Delta}|e^{s\Delta}h_0+w_2|^{\frac4d}\big(e^{s\Delta}h_0+w_2\big)\,ds.
\end{align*}
Denote $w=w_1-w_2$, then $w$ obeys
\begin{align*}
w(t)=&\int_0^t e^{(t-s)\Delta}\Big[|e^{s\Delta}h_0+w_1|^{\frac4d}\big(e^{s\Delta}h_0+w_1\big)-|e^{s\Delta}h_0+w_2|^{\frac4d}\big(e^{s\Delta}h_0+w_2\big)\Big]\,ds.
\end{align*}
Note that there exists an absolute constant $C>0$ such that
\begin{align*}
\Big||e^{s\Delta}h_0+w_1|^{\frac4d}\big(e^{s\Delta}h_0&+w_1\big)-|e^{s\Delta}h_0+w_2|^{\frac4d}\big(e^{s\Delta}h_0+w_2\big)\Big|\\
&\le C\Big(|e^{s\Delta}h_0|^{\frac4d}+|w_1|^{\frac4d}+|w_2|^{\frac4d}\Big)|w|.
\end{align*}
Then by the positivity of the heat kernel of $e^{t\Delta}$, we have
\begin{align*}
|w(t)|\le C\int_0^t e^{(t-s)\Delta}\Big(|e^{s\Delta}h_0|^{\frac4d}+|w_1(s)|^{\frac4d}+|w_2(s)|^{\frac4d}\Big)|w(s)|\,ds.
\end{align*}
Then we get that for $2\le p<\infty$, $\tau\in (0,\delta]$,
\begin{align*}
\|w\|_{L^p_t((0,\tau); L^2(\R^d))}&
\lesssim \Big\|\int_0^t e^{(t-s)\Delta}|e^{s\Delta}h_0|^{\frac4d}|w(s)|\,ds\Big\|_{L^p_t((0,\tau); L^2(\R^d))}\\
&\quad +\Big\|\int_0^t e^{(t-s)\Delta}\Big(|w_1(s)|^{\frac4d}+|w_2(s)|^{\frac4d}\Big)|w(s)|\,ds\Big\|_{L^p_t((0,\tau); L^2(\R^d))}:=I+II.
\end{align*}
For the first term $I$ above, by using Lemma \ref{lem:Stri-1} and choosing $p$ large enough, we have
\begin{align*}
I=\Big\|\int_0^t e^{(t-s)\Delta}|e^{s\Delta}h_0|^{\frac4d}|w(s)|\,ds\Big\|_{L^p_t((0,\tau); L^2(\R^d))}
\lesssim \Big\||e^{s\Delta}h_0|^{\frac4d}|w(s)|\Big\|_{L^{p_1}_t((0,\tau); L^{r_1}(\R^d))},
\end{align*}
where we have chose $(p_1,r_1)$ that
$$
\frac1{p_1}=\frac2{d+2}+\frac1p;\quad \frac1{r_1}=\frac2{d+2}+\frac12.
$$
(Note that $d>4$ and $p$ is large, we have that $p_1\in (1,2), r_1\in (1,2)$). Hence, by H\"older's inequality, we obtain that
\begin{align*}
I=\Big\|\int_0^t e^{(t-s)\Delta}|e^{s\Delta}&h_0|^{\frac4d}|w(s)|\,ds\Big\|_{L^p_t((0,\tau); L^2(\R^d))}\\
\lesssim &\Big\||e^{s\Delta}h_0|^{\frac4d}|w(s)|\Big\|_{L^{p_1}_t((0,\tau); L^{r_1}(\R^d))}\\
\lesssim & \big\|e^{s\Delta}h_0\big\|_{L^{\frac{2(d+2)}{d}}_{tx}\big((0,\tau)\times \R^d\big)}^{\frac4d}\big\|w\big\|_{L^{p}_t((0,\tau); L^2(\R^d))}.
\end{align*}
For the second term $II$ above, by using Lemma \ref{lem:max-Lp},
\begin{align*}
\Big\|\int_0^t e^{(t-s)\Delta}\Big(|w_1(s)&|^{\frac4d}+|w_2(s)|^{\frac4d}\Big)|w(s)|\,ds\Big\|_{L^p_t((0,\tau); L^2(\R^d))}\\
\lesssim & \Big\|(-\Delta)^{-1}\Big(\big(|w_1(s)|^{\frac4d}+|w_2(s)|^{\frac4d}\big)|w(s)|\Big)\Big\|_{L^{p}_t((0,\tau); L^2(\R^d))}.
\end{align*}
Since $d>4$, by Sobolev's embedding, we further have
\begin{align*}
II=\Big\|\int_0^t e^{(t-s)\Delta}\Big(|w_1(s)&|^{\frac4d}+|w_2(s)|^{\frac4d}\Big)|w(s)|\,ds\Big\|_{L^p_t((0,\tau); L^2(\R^d))}\\
\lesssim & \big\|\big(|w_1(s)|^{\frac4d}+|w_2(s)|^{\frac4d}\big)|w(s)|\big\|_{L^{p}_t((0,\tau); L^\frac{2d}{d+4}(\R^d))}\\
\lesssim & \Big(\|w_1\|_{L^\infty_t((0,\tau); L^2(\R^d))}^{\frac4d}+\|w_1\|_{L^\infty_t((0,\tau); L^2(\R^d))}^{\frac4d}\Big)\|w\|_{L^{p}_t((0,\tau); L^2(\R^d))}.
\end{align*}
Collecting  the estimates above, we obtain that
\begin{align}
\|w\|_{L^{p}_t((0,\tau); L^2(\R^d))}
\lesssim \rho(\tau)\cdot \|w\|_{L^{p}_t((0,\tau); L^2(\R^d))}, \label{22.38}
\end{align}
where
$$
\rho(\tau)=\big\|e^{s\Delta}h_0\big\|_{L^{\frac{2(d+2)}{d}}_{tx}\big((0,\tau)\times \R^d\big)}^{\frac4d}+\|w_1\|_{L^\infty_t((0,\tau); L^2(\R^d))}^{\frac4d}+\|w_2\|_{L^\infty_t((0,\tau); L^2(\R^d))}^{\frac4d}.
$$
By \eqref{1.14} and Lemma \ref{lem:Stri}, we have
$$
\big\|e^{s\Delta}h_0\big\|_{L^{\frac{2(d+2)}{d}}_{tx}\big((0,\tau)\times \R^d\big)} \to 0, \quad \mbox{when } \tau\to 0.
$$
Further, since $w_1,w_2\in C([0,\delta], L^2(\R^d))$, we get
$$
\lim\limits_{\tau\to0} \rho(\tau)\to 0.
$$
Hence, choosing $\tau$ small enough and from \eqref{22.38}, we must have that $w\equiv 0$ on $t\in [0,\tau)$. By iteration again, we must have $w_1\equiv w_2$ on $[0,\delta]$.
Thus we have proved Theorem \ref{thm:main1}.

\section{The proof of Theorem \ref{thm:main2} and Theorem \ref{thm:main3}}

\subsection{The proof of Theorem \ref{thm:main2}}
First, \underline{we deal with the global existence of the Cauchy} \underline{problem \eqref{heat}.}  In  the focusing case (i.e. $\mu=1$), if $\|h_0\|_{\dot  H^{-\gamma_0}(\R^d)}\ll 1$, then from Lemma \ref{lem:local} we have the global existence of the solution.
In the defocusing case (i.e. $\mu=-1$), note that $h=v_L+w$ and
\begin{align}
\big\|v_L(t)\big\|_{L^2(\R^d)}=&\big\|e^{-t|\xi|^2}\widehat{v_0}(\xi)\big\|_{L^2_\xi(\R^d)}\notag\\
\lesssim& \big\|e^{-t|\xi|^2}|\xi|^{\gamma_0}\big\|_{L^\infty_\xi(\R^d)}\|v_0\|_{\dot H^{-\gamma_0}}
\lesssim t^{-\frac{\gamma_0}{2}}\|h_0\|_{\dot H^{-\gamma_0}}.\label{9.26}
\end{align}
Hence, from Lemma \ref{lem:local} again, we have $h(\delta)\in L^2(\R^d)$. Let $I=[\delta,T^*)$ be the maximal lifespan of the solution $h$ of the Cauchy problem \eqref{heat} from the time $t_0=\delta$. Then from the $L^2$ estimate of the solution ( by doing the inner product with $h$ in \eqref{heat}), we have
\begin{align*}
\sup\limits_{t\in I}\|h\|_{L^2}^2+\|\nabla h\|_{L^2_{tx}(I\times \R^d)}^2\le \|h(\delta)\|_{L^2}^2.
\end{align*}
This gives the boundedness of both $\big\|h\big\|_{L^\frac{2(2+d)}{d}_{tx}(I\times\R^d)}$ and $\big\|w\big\|_{L^\frac{2(2+d)}{d}_{tx}(I\times\R^d)}$. Thus  by the global criteria given in  Lemma \ref{lem:local}, we have $T^*=+\infty$.

\subsection{The proof of Theorem \ref{thm:main3}}

Next, \underline{we consider the $L^2$ estimate of the solution,} i.e. the following point-wise decay estimate:
 \begin{equation}\label{decay'}
\|h(t)\|_{L^2}\lesssim C(\|h_0\|_{\dot H^{-\gamma_0}})\ t^{-\frac{\gamma_0}{2}},  \ \ t>0,
\end{equation}
 where $C\big(\|h_0\|_{\dot H^{-\gamma_0}}\big)=O\big(\|h_0\|_{\dot H^{-\gamma_0}}^\kappa\big)$ for some $\kappa>0$. We emphasize that the same notation $C\big(\|h_0\|_{\dot H^{-\gamma_0}}\big)$ is often used in the sequel,  but the power $\kappa$ may change in the different places.  To do this, we divide the following two cases to prove the \eqref{decay'}.

\subsubsection{The focusing case ($\mu=1$)}
When $t\le 1$, it follows from \eqref{9.26} and Lemma \ref{lem:local}, that
\begin{align}\label{local-decay}
\|h(t)\|_{L^2}\le C\big(\|h_0\|_{\dot H^{-\gamma_0}}\big)\ t^{-\frac{\gamma_0}{2}} , \quad  \mbox{ for any } t\in (0,1].
\end{align}
For $t\ge1$, we first claim that  for any $t\le T$ and any fixed $T>1$,
\begin{align}
\|h(t)\|_{L^2(\R^d)}
\le
t^{-\frac{\gamma_0}{2}} C\big(\|h_0\|_{\dot H^{-\gamma_0}}\big)\big(\|h\|_{X(T)}+1\big),\label{claim1}
\end{align}
where
\begin{align}\label{Norm}
\|h\|_{X(T)}=\sup\limits_{t\in[0,T]}\Big(t^\frac{\gamma_0}{2}\|h(t)\|_{L^2(\R^d)}\Big).
\end{align}
Indeed, by Duhamel's formula,  we have
\begin{align*}
\|h(t)\|_{L^2(\R^d)}\le \big\|e^{t\Delta}h_0\big\|_{L^2(\R^d)}+\Big\|\int_0^t e^{(t-s)\Delta}|h(s)|^\frac4dh(s)\,ds\Big\|_{L^2(\R^d)}.
\end{align*}
Similar as \eqref{9.26}, we have
$$
\big\|e^{t\Delta}h_0\big\|_{L^2(\R^d)}\lesssim t^{-\frac{\gamma_0}{2}}\|h_0\|_{\dot H^{-\gamma_0}}.
$$
Then using the estimate above, we further have
\begin{align}
\|h(t)\|_{L^2(\R^d)}
\lesssim & t^{-\frac{\gamma_0}{2}}\|h_0\|_{\dot H^{-\gamma_0}}+ \Big\|\int_0^t e^{(t-s)\Delta}|h(s)|^\frac4dh(s)\,ds\Big\|_{L^2(\R^d)}\notag\\
\lesssim & t^{-\frac{\gamma_0}{2}}\|h_0\|_{\dot H^{-\gamma_0}}+ \Big\|\int_0^{\frac t2}e^{(t-s)\Delta}|h(s)|^\frac4dh(s)\,ds\Big\|_{L^2(\R^d)}\notag\\
&\qquad\qquad+\Big\|\int_{\frac t2}^t e^{(t-s)\Delta}|h(s)|^\frac4dh(s)\,ds\Big\|_{L^2(\R^d)}.\label{7.09-0606}
\end{align}
\underline{For the second term in \eqref{7.09-0606}}, we denote
$$
p_*=1, \mbox{ when } 2\le d< 4;\quad p_*=\frac{2d}{d+4}, \mbox{ when } d\ge 4.
$$

First consider the cases $d=2,3$.  By the decay estimate \eqref{Linear-decay} we have
\begin{align}\label{eq.4.7}
\Big\|\int_0^{\frac t2} & e^{(t-s)\Delta}|h(s)|^\frac4dh(s)\,ds\Big\|_{L^2(\R^d)}\notag\\
\lesssim &
\int_0^{\frac t2} |t-s|^{-\frac d2(\frac1{p_*}-\frac12)} \Big\||h(s)|^\frac4dh(s)\Big\|_{L^{p_*}(\R^d)}\,ds\notag\\
\lesssim  &
t^{-\frac d4}\left[\int_0^{\frac12} \big\|h(s)\big\|_{L^{\frac{d+4}{d}}(\R^d)}^{\frac{d+4}{d}}\,ds
+\int_{\frac12}^{\frac t2} \big\|h(s)\big\|_{L^{\frac{d+4}{d}}(\R^d)}^{\frac{d+4}{d}}\,ds\right]=(I)+(II).
\end{align}
For the term $(I)$, by interpolation and H\"older's inequality, we  obtain
\begin{align}\label{eq.4.8}
(I)\le t^{-\frac d4}&\Big(\int_0^{\frac12} \big\|h(s)\big\|_{L^{\frac{2(d+2)}{d}}(\R^d)}^{\frac{2(d+2)}{d}}\,ds\Big)^{1-d/4}\Big(\int_0^{\frac12} \big\|h(s)\big\|_{L^2(\R^d)}^2\,ds\Big)^{d/4}\notag\\
&\lesssim t^{-\frac{d}4 }C\big(\|h_0\|_{\dot  H^{-\gamma_0}(\R^d)}\big)\Big(\int_0^{\frac12} s^{-\gamma_0}\Big)^{d/4}\notag\\
&\lesssim t^{-\frac{d}4 }C\big(\|h_0\|_{\dot  H^{-\gamma_0}(\R^d)}\big),
\end{align}
where we also use the facts $ h\in L^{\frac{2(d+2)}{d}}_{tx}(R^+\times \R^d)$ and the short-time decay estimate \eqref{local-decay}.

For the term $(II)$, by Gagliardo-Nirenberg's inequality ( see e.g. Tao \cite{tao} ), we have that
\begin{align}\label{eq.4.9}
(II)\lesssim t^{-\frac{d}4 }\int_{\frac12}^{\frac t2}\big\|\nabla h(s)\big\|_{L^2(\R^d)}^{\frac4d+1-\theta_d}\big\| h(s)\big\|_{L^2(\R^d)}^{\theta_d}  \,ds,
\end{align}
where  $\theta_2=2$ for $d=2$ and $\theta_3=\frac{11}6$ for $d=3$.

Note that by Lemma \ref{lem:local},
\begin{align}\label{mass}
\sup\limits_{t\in [\frac12,+\infty)}\|h\|_{L^2}^2+\|\nabla h\|_{L^2_{tx}([\frac12,+\infty)\times \R^d)}^2\lesssim C\big(\|h_0\|_{\dot  H^{-\gamma_0}(\R^d)}\big)
\end{align}
holds for  either $\mu=-1$ (the defocusing case) or $\mu=1$ (the focusing case) with $\|h_0\|_{\dot  H^{-\gamma_0}(\R^d)}\ll 1.$
 Hence, when $d=2$, $\theta_2=2$, then  by the \eqref{eq.4.9} and \eqref{mass},
 \begin{align}\label{eq.4.11}
 (II)\lesssim & t^{-\frac12}\int_{\frac12}^{\frac t2}\big\|\nabla h(s)\big\|_{L^2(\R^d)}\big\| h(s)\big\|_{L^2(\R^d)}^2  \,ds\notag\\
 \lesssim & t^{-\frac 12}\left(\int_{\frac12}^{\frac t2}\big\|\nabla h(s)\big\|_{L^2(\R^d)}^2\,ds\right)^\frac12
 \left(\int_{\frac12}^{\frac t2}\big\| h(s)\big\|_{L^2(\R^d)}^{4}\,ds\right)^\frac12\notag\\
 \lesssim & t^{-\frac{\gamma_0}2}C\big(\|h_0\|_{\dot  H^{-\gamma_0}(\R^d)}\big)\|h\|_{X(T)}.
 \end{align}
 While for $d=3$,  we have $\theta_3=\frac{11}6$ and
  \begin{align}\label{eq.4.12}
 (II)\lesssim &t^{-\frac 34}\int_{\frac12}^{\frac t2}\big\|\nabla h(s)\big\|_{L^2(\R^d)}^{\frac12}\big\| h(s)\big\|_{L^2(\R^d)}^{\frac{11}{6}}\,ds\notag\\
\lesssim & t^{-\frac 34}\left(\int_{\frac12}^{\frac t2}\big\|\nabla h(s)\big\|_{L^2(\R^d)}^2\,ds\right)^\frac14
 \left(\int_{\frac12}^{\frac t2}\big\| h(s)\big\|_{L^2(\R^d)}^{\frac{44}{18}}\,ds\right)^\frac34\notag\\
 \lesssim & t^{-\frac12 \gamma_0}C\big(\|h_0\|_{\dot  H^{-\gamma_0}(\R^d)}\big)\|h\|_{X(T)}.
 \end{align}
Combining \eqref{eq.4.8}, \eqref{eq.4.11} and \eqref{eq.4.12}, this yields that when $d=2,3$,
 \begin{align*}
\int_0^{\frac t2} |t-s|^{-\frac d4(\frac1{p_*}-\frac12)} \Big\||h(s)|^\frac4dh(s)\Big\|_{L^{p_*}(\R^d)}\,ds
&\lesssim
(I)+(II)\\
&\lesssim
t^{-\frac{\gamma_0}2}C\big(\|h_0\|_{\dot  H^{-\gamma_0}(\R^d)}\big)\big(1+\|h\|_{X(T)}\big).
\end{align*}
Now we consider the estimates when $d\ge 4$. Done similarly as above, we can obtain
 \begin{align*}
\int_0^{\frac t2} |t-s|^{-\frac d2(\frac1{p_*}-\frac12)} \Big\||h(s)|^\frac4dh(s)\Big\|_{L^{p_*}(\R^d)}\,ds
=&
\int_0^{\frac t2} |t-s|^{-1} \big\|h(s)\big\|_{L^2(\R^d)}^{\frac4d+1}\,ds\\
\lesssim  &
t^{-1}\int_0^{\frac t2} \big\|h(s)\big\|_{L^2(\R^d)}^{\frac4d+1}\,ds\\
\lesssim  &
t^{-\frac12 \gamma_0}C\big(\|h_0\|_{\dot  H^{-\gamma_0}(\R^d)}\big)(1+\|h\|_{X(T)}).
\end{align*}
Combining with the estimates above, we get that for any $d\ge 2$,
\begin{align}
\Big\|\int_0^{\frac t2} & e^{(t-s)\Delta}|h(s)|^\frac4dh(s)\,ds\Big\|_{L^2(\R^d)}
\le
t^{-\frac12 \gamma_0}C\big(\|h_0\|_{\dot  H^{-\gamma_0}(\R^d)}\big)\big(1+\|h\|_{X(T)}\big).
\label{2020-15.21}
\end{align}

\underline{For the third term in \eqref{7.09-0606}}, by positivity of the heat flow $e^{t\Delta}$, we have
\begin{align*}
\Big\|\int_{\frac t2}^t e^{(t-s)\Delta}|h(s)|^\frac4dh(s)\,ds\Big\|_{L^2(\R^d)}
\lesssim &
t^{-\frac12\gamma_0}\Big\|\int_0^t e^{(t-s)\Delta} s^{\frac12\gamma_0}|h(s)|^{\frac4d+1}\,ds\Big\|_{L^2(\R^d)}.
\end{align*}
Then by Lemma \ref{lem:Stri-1} and Sobolev's embedding estimates, we get that
\begin{align*}
\Big\|\int_{\frac t2}^t e^{(t-s)\Delta}|h(s)|^\frac4dh(s)\,ds\Big\|_{L^2(\R^d)}
\lesssim &
t^{-\frac12\gamma_0}\Big\| s^{\frac12\gamma_0}|h(s)|^{\frac4d+1}\Big\|_{L^{\frac d2}_tL^{\frac{2d^2}{d(d+4)-8}}_x(\R^+\times\R^d)}\\
\lesssim &
t^{-\frac12\gamma_0} \big\|\nabla h(s)\big\|_{L^2_{tx}(\R^+\times\R^d)}^{\frac4d}\Big\| s^{\frac12\gamma_0}h(s)\Big\|_{L^\infty_tL^2_x(\R^+\times\R^d)}\\
\lesssim &
 t^{-\frac12 \gamma_0}C\big(\|h_0\|_{\dot  H^{-\gamma_0}(\R^d)}\big)\|h\|_{X(T)},
\end{align*}
which along with \eqref{7.09-0606} and \eqref{2020-15.21},  gives the claim \eqref{claim1}.

Now with \eqref{claim1} in hand. If $\|h_0\|_{\dot H^{-\gamma_0}}\ll 1$, then $$C\big(\|h_0\|_{\dot  H^{-\gamma_0}(\R^d)}\big)=O\big(\|h_0\|_{\dot H^{-\gamma_0}}^\alpha\big)\ll 1,$$ which leads to $\|h\|_{X(T)}\lesssim C\big(\|h_0\|_{\dot  H^{-\gamma_0}(\R^d)}\big)$ from the \eqref{claim1} for any $T>0$.  Thus,  it immediately gives the $L^2$-estimate \eqref{decay'} in the focusing case. For the defocusing case, however, it is not sufficient to establish the $L^2$-estimate with a large data $h_0$, which requires us more work.

\subsubsection{The defocusing case}
As described above, the estimate \eqref{claim1} is not enough to obtain the desired decay result in the defocusing case. To establish the desired result, we introduce  a different approach.
Similarly as above, when $t\le 1$, it follows from \eqref{9.26} and Lemma \ref{lem:local}, that
$$\|h(t)\|_{L^2}\lesssim C\big(\|h_0\|_{\dot H^{-\gamma_0}}\big)\ t^{-\frac{\gamma_0}{2}}, \quad  \mbox{ for any } t\in (0,1].
$$
\underline{So we only need to consider $t\ge 1$}. In the sequel,  we only give the proof when $d=2$ (the general cases can be obtained by suitably adjusting the parameters as before).

Let us begin with establishing a modified estimate of \eqref{claim1} as follows:
\begin{lem}\label{lem:weak}
Let $d=2$, $N(t)$ be a positive time-dependent function and $\alpha\in (0,\gamma_0]$. Then
\begin{align}
\big\|P_{\le N(t)}h(t)\big\|_{L^2(\R^2)}
\le
t^{-\frac{\alpha}{2}}C\big(\|h_0\|_{\dot H^{-\gamma_0}}\big)\Big(1+N(t)t^{\frac12\alpha}+N(t)t^{\frac{1}2}\ \|h\|_{Y_\alpha(T)}\Big),\label{claim2}
\end{align}
where $1\le t\le T$ for any fixed $T>1$ and $$
\|h\|_{Y_\alpha(T)}=\sup\limits_{t\in[0,T]}\Big(t^\frac{\alpha}{2}\|h(t)\|_{L^2(\R^2)}\Big).
$$
\end{lem}
\begin{proof}
By Duhamel's formula and Bernstein's inequality, we have that
\begin{align}\label{eq.4.15}
\big\|P_{\le N(t)}h(t)\big\|_{L^2(\R^2)}
\lesssim & t^{-\frac{{\gamma_0}}{2}}\|P_{\le N(t)}h_0\|_{\dot H^{-\gamma_0}}+ \Big\|P_{\le N(t)}\int_0^t e^{(t-s)\Delta}|h(s)|^2h(s)\,ds\Big\|_{L^2(\R^2)}\notag\\
\lesssim & t^{-\frac{{\gamma_0}}{2}}\|h_0\|_{\dot H^{-\gamma_0}}+\int_0^t \Big\|\big(||P_{\le N(t)}|h(s)|^2h(s)\big)\Big\|_{L^2(\R^2)}ds\notag\\
\lesssim & t^{-\frac{\gamma_0}{2}}\|h_0\|_{\dot H^{-\gamma_0}}+N(t)\int_0^t \Big\|\big(|h(s)|^2h(s)\big)\Big\|_{L^1(\R^2)}\,ds\notag\\
=& t^{-\frac{\gamma_0}{2}}\|h_0\|_{\dot H^{-\gamma_0}}+N(t)\int_0^t \big\|h(s)\big\|_{L^3(\R^2)}^3\,ds.
\end{align}
Note that by $L^2$ estimate, the inequality \eqref{mass} still holds in the defocusing case without smallness assumption.
So treated similarly as above ( See the \eqref{eq.4.7}-\eqref{eq.4.12} for $d=2$ ), we can obtain
\begin{align*}
\int_0^t \big\|h(s)\big\|_{L^3(\R^2)}^3\,ds
=&\int_0^1 \big\|h(s)\big\|_{L^3(\R^2)}^3\,ds+\int_1^t \big\|h(s)\big\|_{L^3(\R^2)}^3\,ds\\
\lesssim &\int_0^1 \big\|h(s)\big\|_{L^3(\R^2)}^3\,ds+\int_1^t \big\|\nabla h(s)\big\|_{L^2(\R^2)}\big\|h(s)\big\|_{L^2(\R^2)}^2\,ds\\
\lesssim & \Big(\int_0^{1} \big\|h(s)\big\|_{L^{4}(\R^2)}^{4}\,ds\Big)^{\frac12}\Big(\int_0^{1} \big\|h(s)\big\|_{L^2(\R^2)}^2\,ds\Big)^{\frac12}\\
&+\ \left(\int_1^t \big\|\nabla h(s)\big\|_{L^2(\R^2)}^2\,ds\right)^\frac12\left(\int_1^t \big\| h(s)\big\|_{L^2(\R^2)}^4\,ds\right)^\frac12\\
\lesssim & C\big(\|h_0\|_{\dot H^{-\gamma_0}}\big)
+t^{\frac12-\frac12\alpha}C\big(\|h_0\|_{\dot H^{-\gamma_0}}\big)\|h\|_{Y_\alpha(T)},
\end{align*}
which combining the \eqref{eq.4.15} immediately gives the estimate \eqref{claim2}.
\end{proof}

Next, we can establish the following weak decay result.
\begin{prop}\label{prop:weak}
Let $\alpha\in (0,\gamma_0]$ be suitable small, then
\begin{align}
\|h(t)\|_{L^2(\R^2)}\le C\big(\|h_0\|_{\dot H^{-\gamma_0}}\big) t^{-\frac{\alpha}{2}}, \quad  \mbox{ for any } t\ge 1.\label{alpha-decay}
\end{align}
\end{prop}

\begin{proof}Consider the mass estimate:
\begin{align*}
\frac{d}{dt}\|h(t)\|_{L^2}^2+\|\nabla h(t)\|_{L^2}^2\le 0.
\end{align*}
This implies that
\begin{align*}
\frac{d}{dt}\|h(t)\|_{L^2}^2+N(t)^2\left\|P_{\ge N(t)}h(t)\right\|_{L^2}^2\le 0,
\end{align*}
and thus
\begin{align*}
\frac{d}{dt}\|h(t)\|_{L^2}^2+N(t)^2\big\|h(t)\big\|_{L^2}^2\le N(t)^2\left\|P_{\le N(t)}h(t)\right\|_{L^2}^2.
\end{align*}
Then applying Lemma \ref{lem:weak}, we get that
\begin{align*}
\frac{d}{dt}\|h(t)\|_{L^2}^2+N(t)^2\big\|h(t)\big\|_{L^2}^2\le N(t)^2t^{-\alpha}C\big(\|h_0\|_{\dot H^{-\gamma_0}}\big)\big(1+N(t)t^{\frac12\alpha}+N(t)t^{\frac12}\|h\|_{Y_\alpha(T)}\big)^2.
\end{align*}
Now we set $N(t)$ such that
$$
N(t)^2=2\alpha t^{-1},
$$
then we obtain that
\begin{align*}
\frac{d}{dt}\left(t^{2\alpha}\|h(t)\|_{L^2}^2\right)\le 2\alpha t^{-1+\alpha}C\big(\|h_0\|_{\dot H^{-\gamma_0}}\big)\big(1+2\alpha\|h\|_{Y_\alpha(T)}^2\big).
\end{align*}
This gives that for any $t\ge 1$,
\begin{align*}
t^{2\alpha}\|h(t)\|_{L^2}^2
\le
&
\|h(1)\|_{L^2}^2+2\alpha\int_1^t s^{-1+\alpha}C\big(\|h_0\|_{\dot H^{-\gamma_0}}\big)\big(1+2\alpha\|h\|_{Y_\alpha(T)}^2\big)\,ds\\
\le
&
\|h(1)\|_{L^2}^2+t^{\alpha}C\big(\|h_0\|_{\dot H^{-\gamma_0}}\big)\big(1+2\alpha\|h\|_{Y_\alpha(T)}^2\big).
\end{align*}
This yields that
\begin{align*}
t^{\alpha}\|h(t)\|_{L^2}^2
\le
&
C\big(\|h_0\|_{\dot H^{-\gamma_0}}\big)+2\alpha C\big(\|h_0\|_{\dot H^{-\gamma_0}}\big)\|h\|_{Y_\alpha(T)}^2.
\end{align*}
Therefore, when $\alpha$ is suitable small such that
$$
2\alpha C\big(\|h_0\|_{\dot H^{-\gamma_0}}\big)\le \frac12,
$$
then we obtain that
$$
\|h\|_{Y_\alpha(T)}\le  C\big(\|h_0\|_{\dot H^{-\gamma_0}}\big).
$$
Since the estimate is independent of $T$, this gives  the desired \eqref{alpha-decay}.
\end{proof}

Finally, based on Lemma \ref{lem:weak} and Proposition \ref{prop:weak}, we come to establish the $L^2$-estimate \eqref{decay'} with a large data $h_0$ in the defocusing case.
Repeat the process above, we first claim that
\begin{align}
\big\|P_{\le N(t)}h(t)\big\|_{L^2(\R^2)}
\le
t^{-\frac{\gamma_0}{2}}C\big(\|h_0\|_{\dot H^{-\gamma_0}}\big)\big(1+N(t)t^{\frac12\gamma_0}+N(t)t^{\frac{1-\alpha}2}\|h\|_{Y_{\gamma_0}(T)}\big).\label{claim3}
\end{align}
Indeed, by the estimate in the proof of Lemma \ref{lem:weak}, we have that
\begin{align*}
\big\|P_{\le N(t)}h(t)\big\|_{L^2(\R^2)}
\lesssim & t^{-\frac{\gamma_0}{2}}\|h_0\|_{\dot H^{-\gamma_0}}+N(t)C\big(\|h_0\|_{\dot H^{-\gamma_0}}\big)\\
&\qquad+N(t)\left(\int_1^t \big\|\nabla h(s)\big\|_{L^2(\R^2)}^2\,ds\right)^\frac12\left(\int_1^t \big\| h(s)\big\|_{L^2(\R^2)}^4\,ds\right)^{\frac12}.
\end{align*}
Now using the weak decay estimate \eqref{alpha-decay} Proposition \ref{prop:weak},  we can further obtain
\begin{align*}
\big\|P_{\le N(t)}h(t)\big\|_{L^2(\R^2)}
\lesssim  & t^{-\frac{\gamma_0}{2}}\|h_0\|_{\dot H^{-\gamma_0}}+N(t)C\big(\|h_0\|_{\dot H^{-\gamma_0}}\big)
+N(t)t^{\frac12-\frac12\alpha-\frac{1}{2}\gamma_0}C\big(\|h_0\|_{\dot H^{-\gamma_0}}\big)\|h\|_{Y_{\gamma_0}(T)}\\
\lesssim & t^{-\frac{\gamma_0}{2}}C\big(\|h_0\|_{\dot H^{-\gamma_0}}\big)\Big(1+N(t)t^{\frac12\gamma_0}+N(t)t^{\frac12-\frac12\alpha}\|h\|_{Y_{\gamma_0}(T)}\Big).
\end{align*}
This gives the desired  \eqref{claim3}.

Treated similarly as in the proof of the estimate \eqref{alpha-decay} above, we use the mass estimate and then apply \eqref{claim3}, to obtain
\begin{align*}
\frac{d}{dt}\|h(t)\|_{L^2(\R^2)}^2+N(t)^2\big\|h(t)\big\|_{L^2(\R^2)}^2\le N(t)^2t^{-\gamma_0}C\big(\|h_0\|_{\dot H^{-\gamma_0}}\big)\big(1+N(t)t^{\frac12\gamma_0}+N(t)t^{\frac{1-\alpha}2}\|h\|_{Y_{\gamma_0}(T)}\big)^2.
\end{align*}
At this time,  we set $N(t)$ such that
$$
N(t)^2=2\gamma_0 t^{-1},
$$
then we obtain that
\begin{align*}
\frac{d}{dt}\left(t^{2\gamma_0}\|h(t)\|_{L^2(\R^2)}^2\right)\le 2\gamma_0 t^{-1+\gamma_0}C\big(\|h_0\|_{\dot H^{-\gamma_0}}\big)\big(1+2\gamma_0t^{-\alpha}\|h\|_{Y_{\gamma_0}(T)}^2\big).
\end{align*}
This implies that for any $t\ge 1$,
\begin{align*}
t^{2\gamma_0}\|h(t)\|_{L^2(\R^2)}^2
\le
&
\|h(1)\|_{L^2(\R^2)}^2+2\gamma_0\int_1^t s^{-1+\gamma_0}C\big(\|h_0\|_{\dot H^{-\gamma_0}}\big)\big(1+2\gamma_0s^{-\alpha}\|h\|_{Y_{\gamma_0}(T)}^2\big)\,ds\\
\lesssim
&
\|h(1)\|_{L^2(\R^2)}^2+t^{\gamma_0}C\big(\|h_0\|_{\dot H^{-\gamma_0}}\big)\big(1+\gamma_0 t^{-\alpha}\|h\|_{Y_{\gamma_0}(T)}^2\big).
\end{align*}
Hence we have that
\begin{align*}
t^{\gamma_0}\|h(t)\|_{L^2(\R^2)}^2
\le
&
C\big(\|h_0\|_{\dot H^{-\gamma_0}}\big)+\gamma_0t^{-\alpha} C\big(\|h_0\|_{\dot H^{-\gamma_0}}\big)\|h\|_{Y_{\gamma_0}(T)}^2.
\end{align*}
Now we choose $t_0=t_0(\|h_0\|_{\dot H^{-\gamma_0}})$ is suitable large such that
$$
t_0^{-\alpha} C\big(\|h_0\|_{\dot H^{-\gamma_0}}\big)\le \frac12,
$$
then we obtain that for any $t\ge t_0$,
\begin{align*}
t^{\gamma_0}\|h(t)\|_{L^2(\R^2)}^2
\le
&
C\big(\|h_0\|_{\dot H^{-\gamma_0}}\big)+\frac12\|h\|_{Y_{\gamma_0}(T)}^2.
\end{align*}
Therefore, this last estimate combining with \eqref{mass} gives that for any $t\ge 1$,
\begin{align*}
t^{\gamma_0}\|h(t)\|_{L^2(\R^2)}^2
\le
&
t_0^{\gamma_0}\|h(t_0)\|_{L^2(\R^2)}^2+ C\big(\|h_0\|_{\dot H^{-\gamma_0}}\big)+\frac12\|h\|_{Y_{\gamma_0}(T)}^2\\
\le &
C\big(\|h_0\|_{\dot H^{-\gamma_0}}\big)+\frac12\|h\|_{Y_{\gamma_0}(T)}^2.
\end{align*}
This allows us to  obtain
$$
\|h\|_{Y_{\gamma_0}(T)}\lesssim  C\big(\|h_0\|_{\dot H^{-\gamma_0}}\big).
$$
Since the estimate is independent of $T$, we give that
\begin{align*}
\|h(t)\|_{L^2(\R^2)}\le C\big(\|h_0\|_{\dot H^{-\gamma_0}}\big) t^{-\frac{\gamma_0}{2}}, \quad  \mbox{ for any } t\ge 1.
\end{align*}
Together with the decay estimate when $t\le 1$, we thus finished the proof of  Theorem \ref{thm:main2}.

\end{document}